\RequirePackage{fix-cm}
\documentclass[11pt,a4paper]{article}

\usepackage{graphicx}
\usepackage{amsmath}
\usepackage{amsfonts}
\usepackage{amssymb}
\usepackage{enumerate}
\usepackage{lscape}
\usepackage{longtable}
\usepackage{rotating}
\usepackage{multirow}
\usepackage{color}
\usepackage{url}
\usepackage{subfigure}
\usepackage{rotating}
\newtheorem{theorem}{Theorem}[section]

\usepackage[ruled,vlined,noline,linesnumbered]{algorithm2e}

\newtheorem{definition}{Definition}[section]

\newtheorem{lemma}{Lemma}[section]

\newtheorem{remark}{Remark}[section]

\newtheorem{assumption}{Assumption}[section]

\newenvironment{proof}[1][Proof]{\textbf{#1.} }{\ \rule{0.5em}{0.5em} \vspace{1ex}}
\usepackage{graphicx}
\usepackage{amsfonts}
\usepackage{caption}
\usepackage{subfigure}
\usepackage{amsmath}
\usepackage{amsfonts}
\usepackage{amssymb}
\usepackage{enumerate}
\usepackage[ruled,vlined]{algorithm2e}
\usepackage{url}
\usepackage{pgf,pgfarrows,pgfnodes,pgfautomata,pgfheaps,pgfshade}
\usepackage{booktabs,multirow,lscape}
\usepackage{algorithmic}
\usepackage{pgfplots}
\usetikzlibrary{patterns}
\newcommand{\real}{\mathbb{R}}

\newtheorem{corol}{Corollary}[section]
 \def\emailname{E-mail}%
\def\email#1{\emailname: #1}
 \def\keywordname{{\bfseries Keywords}}%
  \def\andname{and}%

\DeclareMathOperator{\HI}{HI}
\DeclareMathOperator{\Vol}{Vol}

\begin{document}

\title{Worst-case Complexity Bounds of Directional Direct-search Methods for Multiobjective Optimization
}


\author{
A. L. Cust\'{o}dio \thanks{Department of Mathematics,
FCT-UNL-CMA, Campus de Caparica, 2829-516 Caparica, Portugal.
\newline Support for this author was provided by Funda\c c\~ao
para a Ci\^encia e a Tecnologia (Portuguese Foundation for Science
and Technology) under the projects PTDC/MAT-APL/28400/2017 and
UIDB/00297/2020. \email{{\tt alcustodio@fct.unl.pt}}.
}
\and
Y. Diouane\thanks{ISAE-SUPAERO, Universit\'e de Toulouse,
              31055 Toulouse Cedex 4, France. 
              \email{{\tt youssef.diouane@isae.fr}}.
}
\and
R. Garmanjani \thanks{Department of Mathematics, FCT-UNL-CMA, Campus
de Caparica, 2829-516 Caparica, Portugal.
\newline Support for this author was provided by Funda\c c\~ao
para a Ci\^encia e a Tecnologia (Portuguese Foundation for Science
and Technology) under the projects PTDC/MAT-APL/28400/2017 and
UIDB/00297/2020.  \email{{\tt
r.garmanjani@fct.unl.pt}}
}
\and
E. Riccietti \thanks{INP-ENSEEIHT, Universit\'e de Toulouse,
31071 Toulouse Cedex 7, France. \newline Support for this author was provided by TOTAL E\&P.
  \email{{\tt elisa.riccietti@enseeiht.fr}}
  }
}


%

\maketitle

\begin{abstract}
Direct Multisearch is a well-established class of algorithms,
suited for multiobjective derivative-free optimization. In this
work, we analyze the worst-case complexity of this class of
methods in its most general formulation for unconstrained
optimization. Considering nonconvex smooth functions, we show that
to drive a given criticality measure below a specific positive
threshold, Direct Multisearch takes at most a number of iterations
proportional to the square of the inverse of the threshold, raised
to the number of components of the objective function. This number
is also proportional to the size of the set of linked sequences
between the first unsuccessful iteration and the iteration
immediately before the one where the criticality condition is
satisfied. We then focus on a particular instance of Direct
Multisearch, which considers a more strict criterion for accepting
new nondominated points. In this case, we can establish a better
worst-case complexity bound, simply proportional to the square of
the inverse of the threshold, for driving the same criticality
measure below the considered threshold.

\end{abstract}
\bigskip

\begin{center}
\end{center}

\section{Introduction}

Multiobjective optimization is a challenging domain in nonlinear
optimization~\cite{RTMarler_JSArora_2004,MEhrgott_2005}, when
there are different conflicting objectives that need to be
optimized. Difficulties increase if derivatives are not available,
neither can be numerically approximated due to the associated
computational cost or to the presence of
noise~\cite{ALCustodio_MEmmerich_JFAMadeira_2012}. We are then in
the domain of multiobjective derivative-free optimization, which
often appears in problems where the objective function is
evaluated through numerical simulation (for complementary
information on single-objective derivative-free optimization
methods,
see~\cite{TGKolda_RMLewis_VTorczon_2003,ARConn_KScheinberg_LNVicente_2009,CAudet_WHare_2017}).

We are interested in establishing worst-case complexity (WCC)
bounds for directional direct-search, a class of derivative-free
optimization methods, when used for solving unconstrained
multiobjective optimization problems. Each iteration of this class
of algorithms can be divided into a search step and a poll step,
being the former optional. In fact, the convergence properties of
these methods rely on the procedure implemented in the poll
step~\cite{ALCustodio_et_al_2011}. The objective function is
evaluated at a finite set of points, corresponding to directions
with good geometrical properties, scaled by a stepsize parameter.
The decision of accepting or rejecting a new evaluated point is
solely based on the objective function value, no model is built
for the objective function, neither any attempt of estimating
derivatives is considered~\cite{ALCustodio_et_al_2011}. The
criterion for accepting a new evaluated point makes use of the
partial order induced by the concept of Pareto dominance (cf.
Definition~\ref{def:pareto_dominance} below).

In the last decades, there has been a growing interest in
evaluating the performance of optimization algorithms in the
worst-case scenario (see, for
instance,~\cite{CCartis_NIMGould_PhLToint_2012,MDodangeh_LNVicente_2014,JFliege_AIFVaz_LNVicente_2019,RGarmanjani_LNVicente_2012,GNGrapiglia_JYuan_YYuan_2015,SGratton_et_al_2015,JKonecny_PRichtarik_2014,YNesterov_2004,LNVicente_2013}).
Usually, the performance of an algorithm is measured by the number
of iterations (or function evaluations) required to drive either
some criticality measure below a given positive threshold or the
function value below the threshold distance to the optimal
function value.

In single-objective nonconvex smooth unconstrained optimization,
Nesterov~\cite[Example 1.2.3]{YNesterov_2004} derived a WCC bound
of $\mathcal{O}\left(\epsilon^{-2}\right)$ for gradient descent
algorithms. A similar bound has been achieved for
trust-region~\cite{SGratton_ASartenaer_PhLToint_2008} and
line-search~\cite{CCartis_PhLSampaio_PhLToint_2015} methods.
Nesterov and Polyak~\cite{YNesterov_BTPolyak_2006} investigated
the use of cubic regularization techniques and then Cartis
\textit{et al}~\cite{CCartis_NIMGould_PhLToint_2011} proposed a
generalization to an adaptive regularized framework using cubics.
For the latter class of methods, by considering second order
algorithmic variants, this bound was improved to
$\mathcal{O}\left(\epsilon^{-3/2}\right)$, including a
derivative-free approach where derivatives are approximated by
finite-differences~\cite{CCartis_NIMGould_PhLToint_2012}.

In the context of single-objective derivative-free optimization,
directional direct-search was the first class of algorithms for
which worst-case complexity bounds were
established~\cite{LNVicente_2013}. The author considered the broad
class of directional direct-search methods which use sufficient
decrease as globalization strategy and established that this class
of algorithms shares, in terms of $\epsilon$, the worst-case
complexity bound of steepest descent for the unconstrained
minimization of a nonconvex smooth function. The complexity of
directional direct-search methods for the optimization of convex
smooth functions has been addressed
in~\cite{MDodangeh_LNVicente_2014}. The bound of
$\mathcal{O}\left(\epsilon^{-2}\right)$ has been improved to
$\mathcal{O}\left(\epsilon^{-1}\right)$, which is identical, in
terms of $\epsilon$, to the one of steepest descent, under
convexity. Complexity results have also been established for the
nonsmooth case. In~\cite{RGarmanjani_LNVicente_2012} a class of
smoothing direct-search methods for the unconstrained optimization
of nonsmooth functions was proposed and it was shown that the
worst-case complexity of this procedure is roughly one order of
magnitude worse than the one for directional direct-search or the
steepest descent method, when applied to smooth functions. Other
types of direct-search methods have been analyzed in the
literature. A probabilistic descent directional direct-search
algorithm has been proposed in~\cite{SGratton_et_al_2015}, which
is characterized by the fact that poll directions only guarantee
descent with a certain fixed probability. The authors establish a
worst-case complexity bound of
$\mathcal{O}\left(\epsilon^{-2}\right)$, which holds with a high
probability. A restricted version of directional direct-search
methods, where no stepsize increase is allowed, along with a
worst-case complexity analysis has also been studied
in~\cite{JKonecny_PRichtarik_2014}. In~\cite{EBergou_et_al_2020},
considering stepsizes independent from the results of each
iteration, the authors propose and analyze a random
derivative-free optimization algorithm which evaluates three
points per iteration and enjoys a worst-case complexity bound of
$\mathcal{O}\left(\epsilon^{-2}\right)$.

As for the worst-case complexity of derivative-based methods for
solving unconstrained multiobjective optimization problems, it has
been shown in~\cite{GNGrapiglia_JYuan_YYuan_2015} that
trust-region methods provide a worst-case complexity bound of
$\mathcal{O}\left(\epsilon^{-2}\right)$. A similar bound has been
derived in~\cite{JFliege_AIFVaz_LNVicente_2019}, and improved to
$\mathcal{O}\left(\epsilon^{-1}\right)$ or
$\mathcal{O}(\log\epsilon^{-1})$, assuming convexity or strong
convexity of the different objective function components.
In~\cite{LCalderon_et_al_2020}, complexity bounds have also been
derived for $p$--order regularization methods, this time under a
H\"{o}lder continuity assumption on the derivatives of the
objective function components.

Regarding the WCC of multiobjective derivative-free optimization
algorithms, a first work~\cite{AZilinskas_2013} showed that an
optimal worst-case algorithm for Lipschitz functions can be
reduced to the computation of centers of balls producing an
uniform cover of the feasible region. A biobjective optimization
algorithm for single variable, twice continuously differentiable
functions was proposed and analyzed
in~\cite{JMCalvin_AZilinskas_2020}. The authors prove that after
$\ell\in \mathbb{N}$ function evaluations, the number of points
that do not belong to the Pareto front is of
$\mathcal{O}(\log(\ell)^2)$.

In this work, we first establish a worst-case complexity bound for
the original Direct Multisearch (DMS)~\cite{ALCustodio_et_al_2011}
class of methods. We show that the DMS algorithm takes at most
$\mathcal{O}\left(|L(\epsilon)|\epsilon^{-2m}\right)$ iterations
for driving a criticality measure below $\epsilon>0$, where
$|L(\epsilon)|$ represents the cardinality of the set of linked
sequences between the first unsuccessful iteration and the
iteration immediately before the one where the criticality
condition is satisfied. We then focus on a particular instance of
this class of algorithms, which considers a more restrictive
condition to accept new nondominated points. For that, we resort
to the standard min-max formulation of the multiobjective
optimization problem, which is widely used in the literature
(e.g., see~\cite{KDVVillacorta_POOliveira_ASoubeyran_2014}
and~\cite[\S 4.2]{GNGrapiglia_JYuan_YYuan_2015} for multiobjective
trust-region methods or~\cite{RTMarler_JSArora_2004} for
additional references). We are able to establish that this
particular instance of DMS enjoys a worst-case complexity bound of
$\mathcal{O}\left(\epsilon^{-2}\right)$ for driving the same
criticality measure below $\epsilon>0$. This bound is identical,
in terms of $\epsilon$, to the one derived for multiobjective
gradient descent~\cite{JFliege_AIFVaz_LNVicente_2019} and
trust-region~\cite{GNGrapiglia_JYuan_YYuan_2015} methods.

With regard to the strategy used to establish the WCC of the
min-max formulation, we highlight that it is not equivalent to a
straightforward application of the technique used for
single-objective optimization to the scalar function obtained by
considering the maximum of the components of the objective
function. In particular, the analysis in~\cite{LNVicente_2013},
which establishes the WCC of directional direct-search for
single-objective optimization, relies on the differentiability of
the objective function, which does not hold when a min-max
formulation is considered. However, the analysis we propose takes
into account the differentiability of the single components.

The remaining of the paper is organized as follows. In
Section~\ref{sec_prelim}, we recall some known results on
multiobjective optimization, which will be used throughout the
paper. The complexity analysis of DMS in its most general form
will be established in Section~\ref{sec:dms}.
Section~\ref{sec_mo_ds} introduces the min-max formulation and
establishes a worst-case complexity bound for it. Some conclusions
are drawn in Section~\ref{sec_conc}.

\section{Preliminaries}\label{sec_prelim}

Let us consider the unconstrained multiobjective derivative-free
optimization problem, defined as
\begin{eqnarray}\label{pr:MOO}
\min\quad & F(x):=\left(f_1(x),\dots,f_m(x)\right)^\top ~~~
\mbox{s.t.} \quad&  x\in\mathbb{R}^n,
\end{eqnarray}where $m\geq 2$,  and each
$f_i:\mathbb{R}^n\rightarrow\mathbb{R}\cup \{+\infty\},\,i\in
I:=\{1,\dots,m\}$ is a continuously differentiable function with
Lipschitz continuous gradient.

When solving a multiobjective optimization problem of
type~(\ref{pr:MOO}), the goal is to identify a \textit{local
Pareto minimizer}~\cite{JFliege_BFSvaiter_2000}, i.e. a point
$x^*\in \mathbb{R}^n$ such that it does not exist another point
$x$ in a neighborhood $\mathcal{N}$ of $x^*$ that dominates $x^*$,
according to Definition~\ref{def:pareto_dominance}.

\begin{definition}[Pareto dominance]\label{def:pareto_dominance}
We say that point $x$ dominates point $x^*$ when $F(x)\prec_F
F(x^*)$, i.e., when $F(x^*)-F(x) \in
\mathbb{R}_{+}^m\setminus\{0\}$.
\end{definition}

Point $x^*$ is then a local Pareto minimizer, if there is a
neighboorhood $\mathcal{N}$ of $x^*$ where $x^*$ is nondominated,
meaning $F(x)\nprec_F F(x^*)$ for all $x\in\mathcal{N}$.

A necessary condition for $x^*\in\mathbb{R}^n$ to be a local
Pareto minimizer is~\cite{JFliege_BFSvaiter_2000}:
\begin{equation}\label{pareto_critical}
\forall d\in\mathbb{R}^n,\, \exists i_d\in I\,:\, \nabla
f_{i_d}(x^*)^{\top}d\geq 0.
\end{equation}
A point satisfying \eqref{pareto_critical} is called a
\textit{Pareto critical point}~\cite{JFliege_BFSvaiter_2000}. We
are then interested in finding Pareto critical points.
In what comes next, $\| \cdot \|$ will denote the vector or matrix $\ell_2$-norm.

Following~\cite{JFliege_BFSvaiter_2000}, to characterize Pareto
critical points, we are going to use, for a given
$x\in\mathbb{R}^n$, the function:
\begin{equation}\label{def_mu_original}
\mu(x):= -\min_{\|d\|\leq 1 }\max_{i\in I} \nabla f_i(x)^{\top}d.
\end{equation}
Fliege and Svaiter~\cite{JFliege_BFSvaiter_2000} showed how some
properties of $\mu(x)$, as reported in the following lemma, relate
to the concept of Pareto critical points. We denote by
$\mathcal{F}(x)$ the solution set of~\eqref{def_mu_original}.

\begin{lemma} \textnormal{\cite[Lemma 3]{JFliege_BFSvaiter_2000}}
   For a given $x\in\mathbb{R}^n$, assume that, for all $i \in I$, the function $f_i$ is continuously differentiable at
   $x$ and let $\mu(x)$ be defined as in \eqref{def_mu_original}. Then:
    \begin{enumerate}
        \item $\mu(x)\geq 0$;
        \item if $x$ is a Pareto critical point of \eqref{pr:MOO} then $0\in \mathcal{F}(x)$ and $\mu(x)=0$;
        \item if $x$ is not a Pareto critical point of \eqref{pr:MOO} then $\mu(x)>0$ and for any $d\in \mathcal{F}(x)$  we have
        \begin{equation*}
        \nabla f_j(x)^\top d\leq \max_{i\in I}\nabla f_i(x)^\top d<0, \;\forall j\in I,
        \end{equation*}
        i.e. $d$ is a descent direction of \eqref{pr:MOO};
        \item the function $x\mapsto \mu(x)$ is continuous;
        \item if $x_k$ converges to $\bar{x}$, $d_k\in\mathcal{F}(x_k)$ and $d_k$ converges to $\bar{d}$, then $\bar{d}\in\mathcal{F}(\bar{x})$.
    \end{enumerate}
\end{lemma}
Function $\mu$ can then be used to provide information about
Pareto criticality of a given point and plays a role similar to
the one of the norm of the gradient in single-objective
optimization.

The following lemma describes the relationship between function
$\mu$ and the norm of the gradient of the components of $F$.
\begin{lemma}\label{lemma:relation:mu:gradF}
For a given $x\in\mathbb{R}^n$ and $\epsilon>0$, assume that, for all $i \in I$, $\nabla f_i(x)$ is well defined.
If $\mu(x)>\epsilon$, then $\|\nabla
f_i(x)\|>\epsilon$, for all $i \in I$.
\end{lemma}
\begin{proof}If $\mu(x)>\epsilon$, then
    $$
    \min_{\|d\|\leq 1 }\max_{i\in I} \nabla f_i(x)^{\top}d<-\epsilon.
    $$
    As a consequence, it exists $\bar{d}$ such that $\|\bar{d}\|\leq1$ and
    $$
    \max_{i\in I} \nabla f_i(x)^{\top}\bar d<-\epsilon,
    $$
    that is, for all $i\in I$, we have
    $$
    - \nabla f_i(x)^{\top}\bar d>\epsilon.
    $$
    Hence, as   $- \nabla f_i(x)^{\top}\bar d\leq \| \nabla f_i(x)\|\|\bar d\|\leq\| \nabla f_i(x)\| $, the thesis follows. 
\end{proof}

In the following, we set $\mu_k := \mu(x_k)$, for all $k\ge 0$.

\section{Direct Multisearch}\label{sec:dms}

Direct Multisearch (DMS) was proposed
in~\cite{ALCustodio_et_al_2011} and generalizes directional
direct-search to multiobjective derivative-free optimization. It
is a general class of methods, that can encompass many algorithmic
variants, depending for instance on the globalization strategy
considered. In this work, we will require sufficient decrease for
accepting new points, via the use of a forcing function $\rho :
]0,+\infty[\longrightarrow ]0,+\infty[$.
Following~\cite{TGKolda_RMLewis_VTorczon_2003}, $\rho$ is a
continuous nondecreasing function, satisfying $\rho(t)/t \to 0$
when $t\downarrow 0$. We consider the typical forcing function
$\rho(t) = ct^p$, with $p>1$, and $c>0$.

DMS makes use of the strict partial order induced by the cone
$\mathbb{R}_{+}^m$. Let $D(L)\subset\mathbb{R}^m$ be the image of
the set of points dominated by a list of evaluated points $L$ and
let $D(L; a)$ be the set of points whose distance in the
${\ell}_{\infty}$-norm to $D(L)$ is no larger than $a>0$.
Algorithm~\ref{alg:dms} corresponds to an instance of the original
DMS~\cite{ALCustodio_et_al_2011} method, which uses a
globalization strategy based on the imposition of a sufficient
decrease condition.

\begin{algorithm}
    \begin{rm}
        \begin{description}
            \item[Initialization] \ \\
            Choose $x_0\in\mathbb{R}^n$ with $f_i(x_0)<+\infty,\forall i
            \in I$, $\alpha_0>0$ an initial stepsize, $0 < \beta_1
            \leq \beta_2 < 1$ the coefficients for  stepsize contraction and
            $\gamma \geq 1$ the coefficient for stepsize expansion. Let
            $\mathcal{D}$ be a set of positive spanning sets. Initialize the
            list of nondominated points and corresponding stepsize
            parameters $L_0=\{ (x_0;\alpha_0 ) \}$. \vspace{1ex}
            \item[For $k=0,1,2,\ldots$] \ \\
            \begin{enumerate}
                \item {\bf Selection of an iterate point:} Order the list $L_k$
                according to some criteria and select the first item $(x;\alpha)
                \in L_k$ as the current iterate and stepsize parameter (thus
                setting $(x_k;\alpha_k)=(x;\alpha)$).\vspace{1ex} \item {\bf Search step:}
                Compute a finite set of points $\{z_s\}_{s\in S}$ and evaluate $F$ at
                each  point in $S$. Compute $L_{trial}$ by removing all
                dominated points, using sufficient decrease, from $L_{k} \cup
                \{(z_s;\alpha_k): s\in S \}$ and selecting a subset of the remaining
                nondominated points. If $L_{trial}
                \neq L_k$ declare the iteration (and the search step) as successful,
                set $L_{k+1}=L_{trial}$, and skip the poll step.\vspace{1ex} \item {\bf Poll
                    step:} Choose a positive spanning set~$D_k$ from the set
                $\mathcal{D}$. Evaluate $F$ at the poll points belonging
                to $\{ x_k+\alpha_k d: \, d \in D_k \}$. Compute $L_{trial}$ by
                removing all dominated points, using sufficient decrease, from $L_{k}
                \cup \{(x_k+\alpha_k d;\alpha_k): d\in D_k \}$ and selecting a subset of
                the remaining nondominated points. If $L_{trial} \neq L_k$ declare the
                iteration (and the poll step) as successful and set $L_{k+1}=L_{trial}$. Otherwise,
                declare the iteration (and the poll step) as unsuccessful and set
                $L_{k+1} = L_k$.
                \vspace{1ex}
                \item {\bf Stepsize parameter update:} If the iteration was successful, then
                maintain  or increase
                the corresponding stepsize parameter, by considering
                $\alpha_{k,new}$ $\in [\alpha_k, \gamma \alpha_k]$.  Replace
                all the new points $(x_k+\alpha_k d;\alpha_k)$ in $L_{k+1}$ by
                $(x_k + \alpha_k d;\alpha_{k,new})$, when success is coming from
                the poll step, or $(z_s;\alpha_k)$ in $L_{k+1}$ by
                $(z_s;\alpha_{k,new})$, when success is coming from the search
                step.
                Replace also $(x_k;\alpha_k)$, if in $L_{k+1}$, by
                $(x_k;\alpha_{k,new})$.\\
                Otherwise, decrease the stepsize parameter, by choosing
                $\alpha_{k,new} \in [\beta_1 \alpha_k, \beta_2 \alpha_k]$, and
                replace the poll pair $(x_k;\alpha_k)$ in $L_{k+1}$ by
                $(x_k;\alpha_{k,new})$. \vspace{1ex}
            \end{enumerate}
        \end{description}
    \end{rm}
    \caption{DMS using sufficient decrease as globalization strategy.}
    \label{alg:dms}
\end{algorithm}

DMS declares an iteration as successful when there are
modifications in the list of nondominated points, meaning that a
new point $x$ was accepted, such that $F(x)\notin
D(L;\rho(\alpha))$, where $\alpha$ represents a
stepsize parameter associated with the current
iteration.

For analyzing the worst-case behavior of the
algorithms presented in this paper, we will need the following
assumptions with regard to the component functions
in~\eqref{pr:MOO}.

\begin{assumption} \label{assum:centralized:pbC11}
    For all $i \in I$, the function $f_i$ is continuously differentiable with
    Lipschitz continuous gradient with constant
    $L_i$. Set $L_{\max}=\max_{i\in I} L_i$.
\end{assumption}

\begin{assumption}\label{assum:lbup}
    For all $i \in I$, the function $f_i$ is lower and upper bounded in $\{x\in\mathbb{R}^n: F(x) \notin D(\{x_0\})\}$, with lower bound $f_i^{\min}$ and
    upper bound $f_i^{\max}$.
    Let $F^{\min}  := \min \{f_1^{\min}, \ldots, f_m^{\min} \}$
    and
    $F^{\max}  := \max \{f_1^{\max}, \ldots, f_m^{\max} \}$.
\end{assumption}

\begin{assumption}\label{assum:compact_set}
    The set $\{x\in\mathbb{R}^n: F(x) \notin
    D(\{x_0\})\}$ is compact.
\end{assumption}

At an unsuccessful iteration of Algorithm~\ref{alg:dms}, none of
the components of the objective function is improved, since no new
point is added to the list. However, the use of Pareto dominance
to accept new points implies that successful iterations do not
necessarily correspond to points that improve all components of
the objective function. In fact, at some successful iterations,
some of these components could increase the corresponding value.
Nevertheless, at every successful iteration, the hypervolume (see
Definition~\ref{def:HI} or~\cite{EZitzler_1999}) corresponding to
the current list of nondominated points always increases.

\begin{definition}\cite[Definition 5.2]{ALCustodio_MEmmerich_JFAMadeira_2012}\label{def:HI}
    The hypervolume indicator (or $S$--metric, from `Size of space covered') for some
    (approximation) set $A \subset \real^m$ and a reference point $r\in \real^m$ that is dominated by all the points in $A$ is defined as:
    \[
    \HI (A) \; := \; \Vol \{ b \in \real^m : b \leq r \wedge \exists a \in A: a \leq b\} \; = \; \Vol \left( \bigcup_{a \in A} [a,r]\right).
    \]
    The inequalities should be understood componentwise, $\Vol(\cdot)$ denotes the Le\-bes\-gue measure of a $m$--dimensional set
    of points, and $[a,r]$ denotes the interval box with lower corner $a$ and upper corner $r$.
\end{definition}

Define $F(L)$ as the image set of a list of points $L$,
i.e,\linebreak $F(L):=\{F(x):(x,\alpha)\in L\}$. We will consider
$r=(f_1^{\max},\ldots,f_m^{\max})$, when computing a hypervolume.
Lemma~\ref{lemma:hypervolume} quantifies the increase in the
hypervolume, associated to successful iterations.

\begin{lemma}\label{lemma:hypervolume} In
Algorithm~\ref{alg:dms}, for a successful iteration $k\geq 0$, we
have
\[
\HI(F(L_{k+1})) - \HI(F(L_{k})) \; \geq \;\left(
\rho(\alpha_k)\right)^m.
\]
\end{lemma}
\begin{proof}
If $k$ is a successful iteration then $L_{k+1}\neq L_k$. Let $x\in
L_{k+1}$ be such that $x\notin L_k$. In this situation,
$F(x)\notin D(L_k;\rho(\alpha_k))$.

Thus $B_{\infty}(F(x),\rho(\alpha_k))\cap D(L_k)=\emptyset$, where
$B_{\infty}(F(x),\rho(\alpha_k))$ represents the $\ell_{\infty}$
ball centered at $F(x)$, with radius $\rho(\alpha_k)$. This means
that at least a hypercube of volume $\rho(\alpha_k)^m$ was added
to the dominated region (the one belonging to
$F(x)+(\mathbb{R}_0^+)^m$).
\end{proof}

Fig.~\ref{fig:hypervolume} illustrates the situation, where for a
biobjective problem, at a successful iteration, the previous
condition is satisfied as an equality. The initial list of
nondominated points is formed by the two points represented by the
dots. The point corresponding to the star, in the interior of the
shaded region, was accepted as a new nondominated point, since it
satisfies the sufficient decrease condition. Thus, the hypervolume
corresponding to the new set of nondominated points has increased
exactly in $\rho(\alpha_k)^2$.

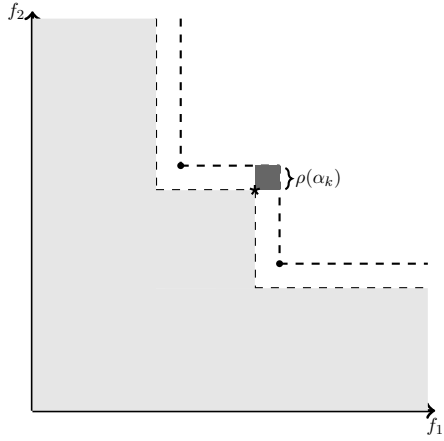
\begin{figure}
\begin{center}
\begin{tikzpicture}[thick,scale=0.65, every node/.style={transform shape}]
        \draw [very thick,->] (0,0) -- (8.15,0) node [below] {$f_1$};
        \draw [very thick,->] (0,0) -- (0,8.15) node [left] {$f_2$};
        \draw[thick,dashed] (2.5,8) -- (2.5,4.5);
        \draw[dashed] (3,8) -- (3,5);
        \draw[dashed] (3,5) -- (5,5);
        \draw[dashed] (5,5) -- (5,3);
        \draw[dashed] (5,3) -- (8,3);
        \draw[thick,dashed] (2.5,4.5) -- (4.5,4.5);
        \draw[thick,dashed] (4.5,4.5) -- (4.5,2.5);
        \draw[thick,dashed] (4.5,2.5) -- (8,2.5);
        \fill (3,5)  circle[radius=2pt];
        \fill (5,3)  circle[radius=2pt];
        \fill[gray!120!white] (4.5,4.5) -- (4.5,5) -- (5,5) -- (5,4.5) -- cycle;
        \fill[gray!20!white] (0,8) -- (2.5,8) -- (2.5,2.5)-- (8,2.5) -- (8,0) -- (0,0)--cycle;
        \fill[gray!20!white](2.49,2.5)-- (2.49,4.5)-- (4.5,4.5)-- (4.5,2.49)--cycle;
        \draw plot[mark=star, mark options={thick, scale=1.5}] coordinates { (4.5,4.5)};
        \node[] at (5.35,4.5){};
                \draw[decoration={brace,raise=4pt},decorate]
        (4.9,4.95) -- node[right=8pt] {$\rho(\alpha_k)$} (4.9,4.52);
\end{tikzpicture}
\caption{\label{fig:hypervolume}Hypervolume increase at a
successful iteration of Algorithm~\ref{alg:dms}.}
\end{center}
\end{figure}

As it is done in classical directional
direct-search~\cite{TGKolda_RMLewis_VTorczon_2003}, we assume that
all positive spanning sets considered by the algorithm include
bounded directions. In multiobjective optimization, the cone of
descent directions for all components of the objective function
can be as narrow as one would like (see
Remark~\ref{rem:newassum}). So, we need to assume density of the
directions at a given limit point, as it is considered in the
convergence analysis of DMS~\cite{ALCustodio_et_al_2011}.

\begin{definition} \label{def:refiningseq}
A subsequence  of iterates $\{x_k\}_{k\in K}$, corresponding to
unsuccessful poll steps, is said to be a refining subsequence if
$\{\alpha_k\}_{k\in K}$ converges to zero.
\end{definition}

The existence of at least one convergent refining subsequence is a
direct consequence of Assumption~\ref{assum:compact_set} and the
use of sufficient decrease for accepting new nondominated points.
Refining directions are limits of normalized poll directions
associated with the refining subsequence. Without loss of
generality, we will assume that all the positive spanning sets
considered have normalized directions.

\begin{assumption} \label{assum:globalconvergence}
Consider Algorithm~\ref{alg:dms} and let $x^*$ be the limit point
of a convergent refining subsequence. Assume that the set of
refining directions associated with $x^*$ is dense in the unit
sphere.
\end{assumption}

We will make use of the following result, which establishes a
relationship between the stepsize parameter at an unsuccessful
iteration of a directional direct-search method and
$\mu_{D_k}(x_k)$, an approximation to $\mu_k=\mu(x_k)$ which only
considers the poll directions.

\begin{lemma}\label{lem:mustep}
    Let Assumption~\ref{assum:centralized:pbC11} hold. Let $k$ be an unsuccessful iteration of Algorithm~\ref{alg:dms},
    $D_k$ be the positive spanning set considered, and $\alpha_k >0$ be the corresponding stepsize.
    Define
    \begin{equation}
\mu_{D_k}(x):= -\min_{d\in D_k, \|d\| \leq 1 }\max_{i\in I} \nabla
f_i(x)^{\top}d.
\end{equation}
Then
    \begin{equation}
  \mu_{D_k}(x_k)  \; \leq \;  \left( \frac{L_{\max}}{2} \alpha_k + \frac{\rho(\alpha_k)}{\alpha_k}\right).\end{equation}
\end{lemma}

\begin{proof}
If iteration $k$ is unsuccessful, then for each direction $d_k\in
D_k$ there is an index $i(d_k)\in I$ such that
\begin{equation*}
f_{i(d_k)}(x_k+\alpha_k d_k) \geq  f_{i(d_k)}(x_k)-\rho(\alpha_k).
\end{equation*}

Hence, for each direction $d_k\in D_k$,
\[0 \leq f_{i(d_k)}(x_k+\alpha_k
d_k)-f_{i(d_k)}(x_k)+\rho(\alpha_k)=\] \[=\int_{0}^{1}\nabla
f_{i(d_k)}(x_k+t\alpha_k d_k)^{\top}\alpha_k  d_k\,dt
        +\rho(\alpha_k)\]
    Adding $-\alpha_k\nabla f_{i(d_k)}(x_k)^{\top}  d_k$ to both sides yields:
    \[
   -\alpha_k\nabla f_{i(d_k)}(x_k)^{\top}  d_k \leq \]
  \[ \leq\int_{0}^{1}\left(\nabla f_{i(d_k)}(x_k+t\alpha_k d_k)^{\top}\alpha_k  d_k -\alpha_k\nabla f_{i(d_k)}(x_k)^{\top} d_k\right)
  \,dt+\rho(\alpha_k)\leq\]
   \[\leq \; \alpha_k^2 \frac{L_{\max}}{2}\| d_k\|^2+\rho(\alpha_k),
    \]
    so that
     \begin{align*}
   -\nabla f_{i(d_k)}(x_k)^{\top}  d_k  \leq \; \alpha_k \frac{L_{\max}}{2}\|  d_k\|^2+\frac{\rho(\alpha_k)}{\alpha_k} .
       \end{align*}
       Then
        \begin{align*}
  \max_{i\in I}   \nabla f_{i}(x_k)^{\top}  d_k  \geq \; -\alpha_k \frac{L_{\max}}{2}\|  d_k\|^2-\frac{\rho(\alpha_k)}{\alpha_k}
       \end{align*}
       This is true for all $d_k\in D_k$ so the thesis holds.
  
\end{proof}

Hereafter, we set $\mu_{D_k} := \mu_{D_k}(x_k)$. In the current
work, WCC bounds will be derived for driving~$\mu_{D_k}$
below~$\epsilon
> 0$. However, the goal is to establish bounds for having
$\mu_{k} \leq \epsilon$. For this purpose, we consider
Assumption~\ref{assum:newassum}. A somehow similar assumption has
already been used within the context of trust-region
derivative-free methods for multiobjective optimization
(see~\cite[Assumption 4.8]{JThomann_GEichfelder_2019}).

\begin{assumption}\label{assum:newassum}
        There exists $\mathcal{C}_1 > 0$ such that
\begin{equation}\label{eq:newassum}
|\mu_{D_k} - \mu_k| \leq \mathcal{C}_1 \mu_{D_k}, \quad \forall k\geq 0.
\end{equation}
\end{assumption}

\begin{remark}\label{rem:newassum}
We note that Assumption~\ref{assum:newassum} requires the
nonnegativity of $\mu_{D_k}$ at every iteration, which may not
hold. In such cases, additional directions could be added to the
positive spanning sets considered as poll directions. Such
procedure is supported by
Assumption~\ref{assum:globalconvergence}.

However, there are cases where Assumption~\ref{assum:newassum} can
be easily satisfied. Let us consider the following biobjective
function:
\[F(x) \; := \; \frac{1}{2} \left(\|x-c_1\|^2, \|x-c_2\|^2\right)^\top,\]
where $c_1 = (-1, 1)^\top$ and $c_2 = -c_1$. This is a biobjective
version of a single-objective variant of the Dennis-Woods
function~\cite{JEDennis_DJWoods_1987} introduced
in~\cite{TGKolda_RMLewis_VTorczon_2003} (see
also~\cite{ARConn_KScheinberg_LNVicente_2009}). The
single-objective function has been used to show that coordinate
search (which considers, at every iteration, the positive spanning
set~$D_k = [I -I]$, where $I$ represents the identity matrix)
stalls at any point $(a,a)^\top$, where $a\neq 0$
(see~\cite{TGKolda_RMLewis_VTorczon_2003,ARConn_KScheinberg_LNVicente_2009}).
We will show that, for this biobjective problem, when considering
the coordinate directions as positive spanning set,
Assumption~\ref{assum:newassum} holds for the majority of points
in~$\real^2$.

When applying DMS, for simplicity, we drop the iteration index $k$
and assume that the set of poll directions is $D= [I -I]$. Then,
for any $x \in \real^2$, one has $ \mu_{D}(x)=-\min
\{x_1+1,x_2+1,-x_1+1,-x_2+1\}.$ Thus, as long as  $x$ belongs to
$\mathcal{B}=\{x \in \real^2: \lvert x_1\rvert > 1 \vee \lvert
x_2\rvert > 1\}$, $\mu_{D}(x)$ will be positive. On the other
hand,
\begin{eqnarray*}
    \mu(x)&=& \max_{\|d\|=1} \min \{-\nabla f_1(x)^Td, -\nabla f_2(x)^Td\} \\
    &\le& \min  \Big\{ \max_{\|d\|=1}-\nabla f_1(x)^Td,
    \max_{\|d\|=1} -\nabla f_2(x)^Td\Big\}.
\end{eqnarray*}
We will show that the assumption holds for all the points in
$\mathcal{B}$ such that $x_1> 1$ and $x_2>1$. For the other points
in $\mathcal{B}$ a similar reasoning can be applied. If $x_1> 1$
and $x_2>1$ one gets that $\mu_{D}(x)=\max \{x_1-1,x_2 - 1\}$ and
\begin{eqnarray*}
    \mu(x) &\le &\min \left\{\mu^1(x),\mu^2(x) \right \}, \mbox{where} \\
    \mu^1(x) &:=& (x_1-1)\sqrt{\frac{(x_1-1)^2}{(x_1-1)^2+(x_2+1)^2}}+(x_2+1)\sqrt{\frac{(x_2+1)^2}{(x_1-1)^2+(x_2+1)^2}},  \\
    \mu^2(x)& :=& (x_1+1)\sqrt{\frac{(x_1+1)^2}{(x_1+1)^2+(x_2-1)^2}}+(x_2-1)\sqrt{\frac{(x_2-1)^2}{(x_1+1)^2+(x_2-1)^2}}.
\end{eqnarray*}
Then, it holds
\begin{eqnarray*}
    \frac{\mu^1(x)}{\mu_{D}(x)}&\leq& \frac{(x_1 -1)+ (x_2+1)}{x_1-1}\leq \frac{x_1+x_2}{x_1-1},\\
    \frac{\mu^2(x)}{\mu_{D}(x)}&\leq& \frac{(x_1 +1)+(x_2-1)}{x_2-1}\leq \frac{x_1+x_2}{x_2-1}.
\end{eqnarray*}
Hence, in this case, inequality~\eqref{eq:newassum} holds with
$\mathcal{C}_1  \ge \max\left\{\frac{x_1+x_2}{x_1-1},
\frac{x_1+x_2}{x_2-1} \right\}-1$. If $x$ is far from the border
of $\mathcal{B}$ this constant will assume reasonable values. For
example, if we assume $2\leq x_1,x_2\leq 5$, then
$\mathcal{C}_1=6$.

Difficulties arise when $x\notin \mathcal{B}$ (in which case
$\mu_{D} \leq 0$) or when $x$ is close to the border of
$\mathcal{B}$ (which makes the constant $\mathcal{C}_1$ large). In
particular, if\linebreak $x=(a,a)^{\top}$ and $a \to 0$, the cone
of descent directions will become as narrow as one would like. In
such cases it is advisable to rotate the set of polling
directions. For example,  to make sure that
Assumption~\ref{assum:newassum} holds at any point $x =
(a,a)^{\top}$ with $\lvert a \rvert< 1$, a possibility would be to
choose, for $l \geq 1$, $2^{l-1}$ maximal positive basis
$\{R_iD\}_{ 0 \le i \le 2^{l-1}-1}$, where
\[ R_i \; = \; \begin{bmatrix}
   \cos (i\theta) & &-\sin (i\theta) \\
    \sin (i\theta) & & \cos (i\theta)
    \end{bmatrix},\]
with $\theta = \frac{\pi}{2^{l}}$ the angle between the generators
of the cone of descent directions. Therefore,  when $a\to 0$ (or
equivalently $\theta \to 0$), Assumption~\ref{assum:newassum} is
satisfied at the cost of increasing the number of function
evaluations.
\end{remark}

In the following theorem, we will derive a bound on the number of
successful iterations required to drive $\mu_k$ below a given
small positive threshold. For each pair of indexes $k_1<k_2$, we
will denote by $U_{k_1}(k_2)$ and $S_{k_1}(k_2)$ the set of
unsuccessful and successful iterations from $k_1$ to $k_2$,
respectively. We will also denote by $k_0$ the index of the first
unsuccessful iteration. We remark that the existence of such index
is ensured by Assumption~\ref{assum:lbup}, from which one can
prove that Algorithm~\ref{alg:dms} generates a sequence of
iterates satisfying $\lim\inf_{k\rightarrow+\infty} \alpha_k =
0$~\cite{ALCustodio_et_al_2011}.
\begin{theorem}\label{thm:bosi1}
Consider the application of Algorithm~\ref{alg:dms} to
problem~(\ref{pr:MOO}), with the choice of forcing function
$\rho(t) = ct^p$, $p>1$, $c>0$. Let
Assumptions~\ref{assum:centralized:pbC11},~\ref{assum:lbup},
and~\ref{assum:newassum} hold. Let $k_0$ be the index of the first
unsuccessful iteration. Given any $\epsilon \in \, ]0,1[$, assume
that $\mu_{k_0} > \epsilon$ and let $j_1$ be the first iteration
after $k_0$ such that $\mu_{j_1+1}\leq \epsilon$. Then, to achieve
$\mu_{j_1+1}\leq \epsilon$ starting from $k_0$,
Algorithm~\ref{alg:dms} takes at most $|S_{k_0} (j_1)| =
\mathcal{O}\left(\epsilon^{-{\frac{pm}{\min(p-1, 1)}}}\right)$
successful iterations.
\end{theorem}
\begin{proof}
Let us assume that $\mu_{k}  > \epsilon$, for $k= k_0, \ldots,
j_1$. Using Assumption~\ref{assum:newassum} we have
\begin{equation}\label{eq:relation:mu:muD}
\epsilon \; < \; \mu_k \; = \; |\mu_k - \mu_{D_k}| + \mu_{D_k} \;
\leq \; (\mathcal{C}_1 + 1) \mu_{D_k}.
\end{equation}
Hence, we obtain $\mu_{D_k} > \epsilon / (1+\mathcal{C}_1)$.

In view of Lemma~\ref{lem:mustep}, for an unsuccessful iteration
$k$, we have
\begin{equation*}
\mu_{D_k}  \; \leq \;  \left( \frac{L_{\max}}{2} \alpha_k  +
\frac{\rho(\alpha_k)}{\alpha_k}\right).\end{equation*}
Thus,
\begin{equation*}
\frac{\epsilon}{1+\mathcal{C}_1}  \; < \;  \left(
\frac{L_{\max}}{2} \alpha_k +
\frac{\rho(\alpha_k)}{\alpha_k}\right),
\end{equation*}
which then implies, when $\alpha_k < 1$,
\[\epsilon \; < \; \mathcal{L}_1 \alpha_k^{\min (p-1, 1)},\]
where $\mathcal{L}_1 = (1+\mathcal{C}_1) \left(
\frac{L_{\max}}{2}+ c \right).$ If $\alpha_k \geq 1$, then
$\alpha_k > \epsilon$. Hence, by combining the two cases
($\alpha_k \geq 1$ and $\alpha_k < 1$) and having $\epsilon < 1$,
when $k$ is an unsuccessful iteration, we have
\begin{equation}\label{eq:lbalpha}
\alpha_k \; > \; \mathcal{L}_2\epsilon^{\frac{1}{\min(p-1, 1)}},
\end{equation}
where $\mathcal{L}_2= \min \left(1 ,
\mathcal{L}_1^{-\frac{1}{\min(p-1,1)}}\right).$

Let $k$ be a successful iteration and $U_{k_0}(k) =\{k_0,
k_1,\ldots, k_u\}$ with~$k_u < k$ be the set of unsuccessful
iterations from $k_0$ to $k$. From Lemma~\ref{lemma:hypervolume}
and by the choice of forcing function,
\begin{align*}
    \HI(F(L_{k+1})) - \HI(F(L_{k_u}))
    \; &\geq \; (k-k_u) (c(\min_{k_{u}+1 \leq t \leq k} \alpha_t)^p)^m \\
    \; &\geq \; |S_{k_u} (k)|( c \beta_ 1^p\alpha_{k_\ell}^p)^m, \quad \mbox{ for some } 0\leq \ell \leq u.
\end{align*}

Notice that the second inequality holds as it is possible to
backtrack from any  iteration $t\in\{k_{u}+1,\ldots,k\}$ to some
previous unsuccessful~$k_{\ell} \in U_{k_0}(k)$ iteration and have
$ \alpha_t\geq \beta_1\alpha_{k_{\ell}}$.

Thus, in view of~\eqref{eq:lbalpha},
\[\HI(F(L_{k+1})) - \HI(F(L_{k_u}))\; \geq \; |S_{k_u} (k)| \left(c\beta_1^p \mathcal{L}_2^p\epsilon^{\frac{p}{\min(p-1, 1)}}\right)^m.\]

By a similar reasoning, for $1\leq i \leq u$, we obtain
\begin{align*}
    \HI(F(L_{k_i})) - \HI(F(L_{k_{i-1}}))
    \; \geq \; |S_{k_{i-1}} (k_i)| \left(c\beta_1^p \mathcal{L}_2^p\epsilon^{\frac{p}{\min(p-1, 1)}}\right)^m.
\end{align*}

Therefore, using the two inequalities above for $k = k_0, \ldots,
j_1$, we obtain
\begin{align*}
    \HI(F(L_{j_1+1})) - \HI(F(L_{k_0})) \; &\geq \; |S_{k_0} (j_1)| \left(c\beta_1^p \mathcal{L}_2^p\epsilon^{\frac{p}{\min(p-1, 1)}}\right)^m.
\end{align*}
Since $\left(F^{\max} - F^{\min}\right)^m \geq \HI(F(L_{j_1+1})) -
\HI(F(L_{k_0})) $, the proof is completed. 
\end{proof}

Now, in order to obtain a bound on the total number of iterations
for driving $\mu_k$ below a given threshold, it remains to find a
bound on the number of unsuccessful iterations. For that, we will
adapt the definition of linked sequences, introduced
in~\cite{GLiuzii_et_al_2016}.

\begin{definition} Consider $\{L_k\}_{k\in\mathbb{N}}$ the
sequence of sets of nondominated points generated by
Algorithm~\ref{alg:dms}. A linked sequence between iterations $i$
and $j$ ($i<j$) is a finite sequence
$\{(x_{l_k},\alpha_{l_k})\}_{k\in\{1,\ldots,n_l\}}$ of maximum
length such that $n_l\leq j-i+1$, $(x_{l_k},\alpha_{l_k})\in
\cup_{r=i}^j L_r$ for all $k\in\{1,\ldots,n_l\}$, and for any
$k\in\{2,\ldots,n_l\}$ and $r\in\{i+1,\ldots,j\}$, the pair
$(x_{l_k},\alpha_{l_k})\in L_r$ is generated by
$(x_{l_{k-1}},\alpha_{l_{k-1}})\in L_{p}$, with $i\leq p<r$.
\end{definition}

Theorem~\ref{thm:bousi1} establishes a bound on the number of
unsuccessful iterations required for driving $\mu_k$ below a given
threshold.

\begin{theorem}\label{thm:bousi1}
Let all the assumptions of Theorem~\ref{thm:bosi1} hold. Then, to
achieve $\mu_k \leq \epsilon$ starting from $k_0$,
Algorithm~\ref{alg:dms} takes at most
\[| U_{k_0}(j_1)| \; \leq \; |\mathbb{L}_{k_0}^{j_1}|\left\lceil - \frac{\log(\gamma) }{ \log(\beta_2) }|S_{k_0}(j_1)| - \frac{ \log(\alpha_{l_1})}{ \log(\beta_2) } +\frac{\log{\left( \beta_1 \mathcal{L}_2
\epsilon^{{\frac{1}{\min(p-1,1)}}}
\right)}}{{\log(\beta_2)}}\right\rceil\] unsuccessful iterations,
where $\alpha_{l_1}$ denotes the stepsize associated to one of the
points in $L_{k_0}$ and $|\mathbb{L}_{k_0}^{j_1}|$ the cardinality
of the set of all linked sequences between $k_0$ and $j_1$.
\end{theorem}
\begin{proof}
Let $\{(x_{l_k},\alpha_{l_k})\}_{k\in\{1,\ldots,n_l\}}$ be a
linked sequence between $k_0$ and $j_1$. Let $S_{k_0}(n_l)$ and
$U_{k_0}(n_l)$ be, respectively, the set of successful and
unsuccessful iterations in the sequence. Assume $U_{k_0}(n_l)\neq
\emptyset$. Since, for any $1\leq k<n_l-1$, either
 $\alpha_{l_{k+1}} \leq \beta_2 \alpha_{l_k}$  (if the iteration is unsuccessful) or $\alpha_{l_{k+1}} \leq \gamma \alpha_{l_k}$  (if the iteration is successful),  we obtain by induction
 \[\alpha_{l_{n_l}} \; \leq \; \alpha_{l_1}\gamma^{|S_{k_0}(n_l)|}
 \beta_2^{|U_{k_0}(n_l)|},\]
 which, in turn, implies from $\log (\beta_2) < 0$
\[|U_{k_0}(n_l)| \; \leq \; - \frac{ \log(\gamma) }{\log(\beta_2) }|S_{k_0}(n_l)|-
\frac{\log(\alpha_{l_1}) }{ \log(\beta_2) } +
\frac{\log(\alpha_{l_{n_l}}) }{ \log(\beta_2) }.\] From $\log
(\beta_2) < 0$ and the lower bound~\eqref{eq:lbalpha} on
$\alpha_k$, we obtain
\[|U_{k_0}(n_l)| \; \leq \; - \frac{ \log(\gamma) }{\log(\beta_2) }|S_{k_0}(n_l)|- \frac{
\log(\alpha_{l_1}) }{ \log(\beta_2) } + \frac{\log\left( \beta_1
\mathcal{L}_2 \epsilon^{{\frac{1}{\min(p-1,1)}}} \right) }{
\log(\beta_2) }.\] The last inequality holds trivially when
$U_{k_0}(n_l)=\emptyset$.

Let $\mathbb{L}_{k_0}^{j_1}$ denote the set of indexes of all
linked sequences between $k_0$ and $j_1$. Then,
\begin{align*}
| U_{k_0}(j_1)| &\leq \; \sum_{l\in \mathbb{L}_{k_0}^{j_1}} |U_{k_0}(n_l)| \\
\; &\leq \; |\mathbb{L}_{k_0}^{j_1}| \max_{l\in \mathbb{L}_{k_0}^{j_1}}|U_{k_0}(n_l)|\\
\; &\leq \; |\mathbb{L}_{k_0}^{j_1}| \max_{l\in
\mathbb{L}_{k_0}^{j_1}} \left[- \frac{ \log(\gamma) }{
\log(\beta_2) }|S_{k_0}(n_l)|- \frac{ \log(\alpha_{l_1}) }{
\log(\beta_2) } + \frac{ \log\left( \beta_1 \mathcal{L}_2
\epsilon^{{\frac{1}{\min(p-1,1)}}} \right) }{ \log(\beta_2)
}\right].
\end{align*}
Since $n_l \le j_1-k_0+1$, one has $|S_{k_0}(n_l)| \le
|S_{k_0}(j_1)|$ for all $l\in \mathbb{L}_{k_0}^{j_1}$. Thus,
\begin{align*}
| U_{k_0}(j_1)| \; &\leq \; |\mathbb{L}_{k_0}^{j_1}| \left[-
\frac{ \log(\gamma) }{ \log(\beta_2) }|S_{k_0}(j_1)|- \frac{
\log(\alpha_{l_1}) }{ \log(\beta_2) } + \frac{ \log\left( \beta_1
\mathcal{L}_2 \epsilon^{{\frac{1}{\min(p-1,1)}}} \right) }{
\log(\beta_2) }\right].
\end{align*}
\end{proof}

In the previous bound, the size of the second term in the sum can
be easily bounded. In fact, from Theorem~\ref{thm:bosi1} we know
that there is a finite number of successful iterations, before
driving $\mu$ below the given threshold. The increase in the
stepsize can be controlled by setting $\gamma = 1$ or by
considering an upper bound for the stepsize itself
(see~\cite{MDodangeh_LNVicente_2014}).

Combining Theorems~\ref{thm:bosi1} and~\ref{thm:bousi1}, it can be
seen that Algorithm~\ref{alg:dms} takes at most
$\mathcal{O}\left(|L(\epsilon)|\epsilon^{-{\frac{pm}{\min(p-1,
1)}}}\right)$ iterations to bring $\mu_k < \epsilon$ for some
$k\geq 0$, where $|L(\epsilon)|$ represents the size of the set of
linked sequences between the first unsuccessful iteration and the
iteration immediately before the one where the criticality
condition is satisfied. The best complexity bound is then derived
by setting $p=2$, which leads to the complexity bound of
$\mathcal{O}\left(|L(\epsilon)|\epsilon^{-2m}\right)$. The WCC
bounds of Algorithm~\ref{alg:dms} in terms of the number of
function evaluations is established in the following corollary.

\begin{corol}\label{cor:totalbound}
Let all the assumptions of Theorem~\ref{thm:bosi1} hold. To
achieve\linebreak$\mu_k < \epsilon$, Algorithm~\ref{alg:dms} takes
at most
$\mathcal{O}\left(|L(\epsilon)|\epsilon^{-{\frac{pm}{\min(p-1,
1)}}}\right)$ iterations (and\linebreak
$\mathcal{O}\left(n|L(\epsilon)|\epsilon^{ -{\frac{pm}{\min(p-1,
1)}}}\right)$ function evaluations). When $p=2$ this bound
is\linebreak $\mathcal{O}\left(|L(\epsilon)|\epsilon^{-2m}\right)$
(and $\mathcal{O}\left(n |L(\epsilon)|\epsilon^{-2m}\right)$
function evaluations).
\end{corol}

\begin{remark}\label{rem:numfunc}
With regard to the number of function evaluations in
Corollary~\ref{cor:totalbound}, since the computational cost of
evaluating each component of the objective function might not be
the same, we have considered the computational cost of $F$ and
count the number of times that it is evaluated at each iteration,
rather than counting separately the number of component functions
evaluations.
\end{remark}

One can see that the bound for DMS, in terms of $\epsilon$, does
not conform with the bound $\mathcal{O}\left(\epsilon^{-2}\right)$
for the gradient descent derived
in~\cite{JFliege_AIFVaz_LNVicente_2019} for
problem~\eqref{pr:MOO}. One of the main reasons behind this
difference is the fact that DMS declares an iteration as
successful if at least one of the components of the objective
function could be improved sufficiently, whereas
in~\cite{JFliege_AIFVaz_LNVicente_2019} the algorithm, which uses
a backtracking approach for determining the right stepsize
parameter, moves to a new point if all the components of the
objective function could be improved sufficiently.  It should also
be noted that DMS will compute an approximation to a complete
local Pareto front, whereas the multiobjective gradient descent
algorithm finds a single Pareto critical point. This explains the
dependence on $|L(\epsilon)|$ for the WCC bounds derived.

If a more demanding criterion is considered to accept new
nondominated points, a complexity bound identical, in terms of
$\epsilon$, to the one derived
in~\cite{JFliege_AIFVaz_LNVicente_2019} for the gradient descent
can be established for DMS. In the next section, we will propose a
direct-search framework, which corresponds to a particular
instance of DMS, and presents a worst-case complexity bound of
$\mathcal{O}\left(n\epsilon^{-2}\right)$, when considering the
number of function evaluations.

\section{A Min-Max Direct-Search Framework for Multiobjective Optimization}\label{sec_mo_ds}
In this section, instead of considering problem~\eqref{pr:MOO}
directly, we use a min-max formulation:
\begin{equation*}
\min_{x\in\mathbb{R}^n} f(x)
\end{equation*}
with
\begin{equation*}
f(x):= \max_{i\in I} f_i(x).
\end{equation*}

Algorithm~\ref{alg:DS:MO} considers a Direct-Search (DS) approach
with a stricter criterion for accepting new nondominated points.
In this case, rather than an approximation to the complete Pareto
front, only one Pareto critical point will be computed for
problem~(\ref{pr:MOO}). For simplicity, the forcing function
$\rho(t) = \frac{c}{2}t^2$, with $c>0$, has been considered and
the (optional) search step has not been included in the
algorithmic description. However, the subsequent results could be
established for a more general setting, such as the one of
Algorithm~\ref{alg:dms}, once that the strict condition for
accepting new nondominated points is used.

\begin{algorithm}
    \begin{rm}
        \begin{description}
            \item[Initialization] \ \\
            Choose $x_0\in\mathbb{R}^n$ with $f_i(x_0)<+\infty,\forall i
            \in I$, $\alpha_0>0$ an initial stepsize, $0 < \beta_1
            \leq \beta_2 < 1$ the coefficients for stepsize contraction and
            $\gamma \geq 1$ the coefficient for stepsize expansion. Let
            $\mathcal{D}$ be a set of positive spanning sets and $c > 0$ a constant used in the sufficient decrease condition.
            \item[For $k=0,1,2,\ldots$] \ \\
            \begin{enumerate}
                \item {\bf Poll step:} Choose a positive spanning set~$D_k$ from the set
                $\mathcal{D}$. Evaluate $F$ at the poll points belonging
                to $\{ x_k+\alpha_k d: \, d \in D_k \}$. If it exists $d_k \in D_k$ such that
        \begin{equation*}
        f(x_k+\alpha_k d_k) < f(x_k) - \frac{c}{2}\alpha_k^2,
        \end{equation*}
        then declare the iteration as successful and set
        $x_{k+1}=x_k+\alpha_k d_k$. Otherwise, declare the
        iteration as unsuccessful and set $x_{k+1}=x_k$.
\vspace{1ex} \item {\bf Stepsize parameter
                    update:} If the iteration was successful then maintain or increase
                the corresponding stepsize parameter, by considering
                $\alpha_{k+1}$ $\in [\alpha_k, \gamma \alpha_k]$.\\
                Otherwise decrease the stepsize parameter, by choosing
                $\alpha_{k+1} \in [\beta_1 \alpha_k, \beta_2 \alpha_k]$.\vspace{1ex}
            \end{enumerate}
        \end{description}
    \end{rm}
     \caption{Min-max DS for multiobjective optimization.}
    \label{alg:DS:MO}
\end{algorithm}

Algorithm~\ref{alg:DS:MO} can be regarded as a particular instance
of Algorithm~\ref{alg:dms}, where no search step is performed, the
list $L_k$ is a singleton, corresponding to the current iterate
and stepsize parameter $(x_k;\alpha_k)$, with a particular choice
of $L_{trial}$ as a subset of the set of computed nondominated
points. Fig.~\ref{fig:dominance} illustrates the latter claim for
a biobjective optimization problem. Consider\linebreak
$F(x_k)=(f_1^k,f_2^k)$ as the objective function value at the
current iterate and $\rho(\alpha_k)$ as the current value of the
forcing function. The shaded region corresponds to the image of
the subset of nondominated points, from which a new iterate can be
selected. This set is a subset of the corresponding set in
Algorithm~\ref{alg:dms} (represented by the hatch-lined area).
Such restriction leads to a better worst-case complexity bound,
comparing to the general formulation of DMS.

\begin{figure}
\begin{center}
        \begin{tikzpicture}[thick,scale=0.65, every node/.style={transform shape}]
        \draw [very thick,->] (0,0) -- (8,0) node [below] {$f_1$};
        \draw [very thick,->] (0,0) -- (0,8) node [left] {$f_2$};
        \draw[dashed] (3,5) --  (3,4.5)  node [right=6pt,pos=0.6] {$\rho(\alpha_k)$};
        \draw[decoration={brace,raise=4pt},decorate](2.92,4.95) -- node[right=3pt] {} (2.92,4.52);
        \fill (3,5)  circle[radius=2pt] node [above right] {$(f_1^k,f_2^k)$};
        \fill[gray!60!white](0,4.5)-- (4.5,4.5)-- (4.5,0)-- (0,0)--cycle;
        \draw[dashed] (0,4.5) -- (4.5,4.5);
        \draw[dashed] (4.5,4.5) -- (4.5,0);
       \pattern[pattern=north east lines] (0,4.5)--(0,0)--(4.5,0)--(4.5,4.5)--cycle;
       \pattern[pattern=north east lines] (0,4.5)--(2.5,4.5)--(2.5,7.8)--(0,7.8)--cycle;
       \pattern[pattern=north east lines] (4.5,0)--(4.5,4.5)--(7.8,4.5)--(7.8,0)--cycle;
        \draw[dashed] (2.5,4.5)--(2.5,7.8);
        \draw[dashed] (4.5,4.5)--(7.8,4.5);
    \end{tikzpicture}
\caption{\label{fig:dominance}Selecting a new nondominated point
in the min-max direct-search framework.}
\end{center}
\end{figure}
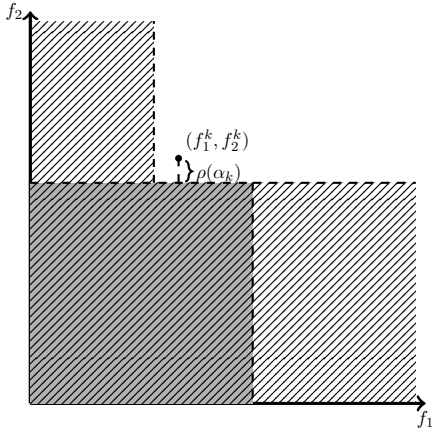

A result similar to Lemma~\ref{lem:mustep} can be established for
Algorithm~\ref{alg:DS:MO}.
\begin{lemma} \label{lemma:centralized:unsucc}
Under Assumption~\ref{assum:centralized:pbC11}, suppose that the
$k$-th iteration of Algorithm~\ref{alg:DS:MO} is unsuccessful. Let
$D_k$ be the positive spanning set considered, and $\alpha_k >0$
be the corresponding stepsize. Then
\begin{equation*}\label{uu}
\mu_{D_k}(x_k) \; \leq \;
\frac{1}{2}\left(L_{\max}\,+c\right)\alpha_k .
\end{equation*}
\end{lemma}

\begin{proof}
If iteration $k$ is unsuccessful then for all directions $d_k\in
D_k$
\begin{equation*}
f(x_k+\alpha_k d_k) \; \geq \;  f(x_k) - \frac{c}{2}\alpha_k^2.
\end{equation*}
Hence, there exists $i(d_k)\in I$ such that
\begin{equation*}
f_{i(d_k)}(x_k+\alpha_k d_k) \; \geq \;   f_{i(d_k)}(x_k) -
\frac{c}{2}\alpha_k^2.
\end{equation*}
The remainder of the proof is similar to the one of
Lemma~\ref{lem:mustep}.
\end{proof}

The following lemma states that the sequence $\sum_{k=0}^{\infty}
\alpha_k^2$ is finite, where $\{\alpha_k\}_{k\geq 0}$ is generated
by Algorithm~\ref{alg:DS:MO}. The proof is identical to the one
of~\cite[Lemma 4.1]{SGratton_et_al_2015}, but we include it for
completeness.
\begin{lemma} \textnormal{\cite[Lemma 4.1]{SGratton_et_al_2015}} \label{lemma:centralized:cvalpha}
Under Assumption~\ref{assum:lbup}, the sequence of
$\{\alpha_k\}_{k\geq 0}$ generated by Algorithm~\ref{alg:DS:MO}
satisfies
\begin{equation*} \label{eq:centralized:cvalpha}
\sum_{k=0}^{\infty} \alpha_k^2 \; \le \; \Omega \; := \;
\frac{\gamma^2}{1-\beta_2^2}\left( \gamma^{-2}\alpha_0^2 +
\frac{2}{c}(f(x_0)-F^{\min})\right),
\end{equation*}
where $\gamma,\beta_2,c$ are defined in Algorithm~\ref{alg:DS:MO}.
\end{lemma}

\begin{proof}
We begin by proving that the series $\sum_{k \in S} \alpha_k^2$ is
finite. To this end, recall that for every successful iteration,
we have
\[f(x_k) - f(x_{k+1}) \ge \frac{c}{2}\alpha_k^2.\]
Moreover, denoting by $S$ the set corresponding to the indexes of
successful iterations, since the iterate does not change between
two successful iterations, we also have for any $K \ge 0$:
\[\sum_{\substack{k \in S \\ k \le K}} f(x_k) - f(x_{k+1}) = \sum_{k \le K} f(x_k) - f(x_{k+1}) = f(x_0)-f(x_{K+1})\le f(x_0) - F^{\min},\]
where the last inequality results from
Assumption~\ref{assum:lbup}. As result, we obtain
\begin{equation*}
f(x_0) - F^{\min} \ge \sum_{k \in S} f(x_k) - f(x_{k+1}) \ge
\sum_{k \in S} \frac{c}{2} \alpha_k^2.
\end{equation*}
Thus, $\sum_{k \in S} \alpha_k^2 \le \frac{2}{c}(f(x_0) -
F^{min})< \infty$.

To analyze the full series, we consider the set
$S:=\{k_0,k_1,k_2,\dots,\}$, where $k_i \ge 0$ is the index of the
$i$-th successful iteration and $k_0=-1$ is an artificial index
corresponding to $\alpha_{-1} = \gamma^{-1}\alpha_0$. With this
notation, given the updating rules on $\alpha_k$, we have that
\begin{equation*}
\sum_{k=0}^{\infty} \alpha_k^2 = \sum_{i=0}^{\infty}
\sum_{k=k_i+1}^{k_{i+1}} \alpha_k^2 \le \sum_{i=0}^{\infty}
\sum_{k=k_i+1}^{k_{i+1}} \gamma^2 \beta_2^{2(k-k_{i}-1)}
\alpha_{k_i}^2  \le
\frac{\gamma^2}{1-\beta_2^2}\sum_{i=0}^{\infty} \alpha_{k_i}^2.
\end{equation*}
We conclude by observing that
\[\sum_{i=0}^{\infty}\alpha_{k_i}^2 = \gamma^{-2}\alpha_0^2 +
\sum_{k \in S} \alpha_k^2 \le \gamma^{-2}\alpha_0^2 +
\frac{2}{c}(f(x_0) - F^{\min}).\]
\end{proof}

Finally, in the main result of this section, we will prove that
Algorithm~\ref{alg:DS:MO} takes at most
$\mathcal{O}(\epsilon^{-2})$ iterations for driving $\mu$ below
$\epsilon>0$. Similarly to Algorithm~\ref{alg:dms},
Algorithm~\ref{alg:DS:MO} cannot be proven globally convergent to
a Pareto critical point for an arbitrary choice of positive
spanning sets as sets of poll directions, as one can easily
present examples where the cone of descent directions, considering
all components of the objective function, can be arbitrarily
narrow (see  Remark~\ref{rem:newassum}).

\begin{theorem} \label{theo:centralized:wccits}
Let Assumptions~\ref{assum:centralized:pbC11},~\ref{assum:lbup},
and~\ref{assum:newassum} hold. For $\epsilon \in \, ]0,1[$, let
$k_{\epsilon}$ be the first iteration index such that
$\mu_{k_{\epsilon}+1} \le \epsilon$. Then,
\begin{equation*} \label{eq:centralized:wccits}
k_{\epsilon} \; \le \;
\frac{2}{c\alpha_0^2}\left(f(x_0)-F^{\min}\right) +
\frac{\Omega(L_{\max}\,+c)^2(\mathcal{C}_1+1)^2}{4\beta_1^2}\epsilon^{-2},
\end{equation*}
where $\Omega$ is defined as in
Lemma~\ref{lemma:centralized:cvalpha}.
\end{theorem}

\begin{proof}
If $k_{\epsilon}=0$ the result trivially holds. Therefore, we
assume in what follows that $k_{\epsilon} > 0$.

For any unsuccessful iteration of index $k \le k_{\epsilon}$, we
have from Lemma~\ref{lemma:centralized:unsucc} that
\begin{eqnarray} \label{eq:centralized:unsucceps:1}
\alpha_k^2 \;  &\ge & \; \frac{4\mu_{D_k}^2}{(L_{\max}\,+c)^2}.
\end{eqnarray}
Since $\mu_k>\epsilon$, in view of~(\ref{eq:relation:mu:muD}), we
have $\mu_{D_k} > \epsilon / (\mathcal{C}_1+1)$. Therefore, using
~(\ref{eq:centralized:unsucceps:1}), we have
\begin{equation}\label{eq:centralized:unsucceps}
\alpha_k^2 \; \ge \;
\frac{4\epsilon^2}{(L_{\max}\,+c)^2(\mathcal{C}_1+1)^2}.
\end{equation}

Considering the updating rules on the stepsize, for any successful
iteration of index $k_{\epsilon}\geq k > j_1$, where $j_1$ is the
index of the first unsuccessful iteration, there exists an index
of an unsuccessful iteration $j(k) \le k$ (with possibly
$j(k)=j_1$) such that $\alpha_k \ge \beta_1 \alpha_{j(k)}$.
Putting this together with~\eqref{eq:centralized:unsucceps}
yields:
\begin{equation*}
\forall k \in S, k_{\epsilon} \; \geq \;  k > j_1, \quad
\alpha_k^2 \ge
\frac{4\beta_1^2\epsilon^2}{(L_{\max}\,+c)^2(\mathcal{C}_1+1)^2},
\end{equation*}where $S$ denotes the set of successful iterations.
Using now the result of Lemma~\ref{lemma:centralized:cvalpha}, we
have:
\begin{equation*}
\Omega \; \ge \; \sum_{k=0}^\infty \alpha_k^2 \; \ge \;
\sum_{k=j_1+1}^{k_{\varepsilon}} \alpha_k^2 \; \ge \;
(k_{\varepsilon}-j_1)\frac{4\beta_1^2\epsilon^2}{(L_{\max}\,+c)^2(\mathcal{C}_1+1)^2}.
\end{equation*}
Thus
\begin{equation*}
k_{\varepsilon}-j_1 \; \le \;
\frac{\Omega(L_{\max}\,+c)^2(\mathcal{C}_1+1)^2}{4\beta_1^2}\epsilon^{-2}.
\end{equation*}
Since $j_1$ is the index of the first unsuccessful iteration, one
can trivially  show that $j_1 \le
\frac{2}{c\alpha_0^2}\left(f(x_0)-F^{\min}\right) $. Then, the
thesis follows. 
\end{proof}

The previous theorem allows us to establish a WCC bound in terms
of the number of function evaluations for
Algorithm~\ref{alg:DS:MO} (see also Remark~\ref{rem:numfunc}).

\begin{corol}
Let all the assumptions of Theorem~\ref{theo:centralized:wccits}
hold. To achieve\linebreak $\mu_k < \epsilon$,
Algorithm~\ref{alg:DS:MO} takes at most $\mathcal{O}\left(n
\epsilon^{-2}\right)$ function evaluations.
\end{corol}

\section{Conclusions}\label{sec_conc}
In this work, we analyzed the worst-case complexity of some
direct-search derivative-free algorithms for unconstrained
multiobjective nonconvex smooth optimization problems. In the case
of Direct Multisearch~\cite{ALCustodio_et_al_2011}, we derived a
complexity bound of $\mathcal{O}(|L(\epsilon)|\epsilon^{-2m})$ for
driving a criticality measure below $\epsilon>0$. We then proposed
a min-max approach to the multiobjective derivative-free
optimization problem, which proved to be a particular instance of
Direct Multisearch, but presented a worst-case complexity bound of
$\mathcal{O}(\epsilon^{-2})$ for driving the same criticality
measure below $\epsilon>0$. This result is identical, in terms of
$\epsilon$, to the one established
in~\cite{JFliege_AIFVaz_LNVicente_2019} for gradient descent,
considering the same class of problems. For the (strongly) convex
case, where all the components of the objective function are
(strongly) convex, it remains as an open question whether similar
complexity bounds to those derived
in~\cite{JFliege_AIFVaz_LNVicente_2019} could be established for
the algorithms considered in this paper.

\section*{Acknowledgments}
{\small The authors would like to thank the three anonymous
referees, whose comments and suggestions much improved the quality
of the paper.}

\bibliographystyle{spmpsci_unsrt}
\bibliography{biblio_ds_moo}

\begin{thebibliography}{10}
\providecommand{\url}[1]{{#1}}
\providecommand{\urlprefix}{URL }
\expandafter\ifx\csname urlstyle\endcsname\relax
  \providecommand{\doi}[1]{DOI~\discretionary{}{}{}#1}\else
  \providecommand{\doi}{DOI~\discretionary{}{}{}\begingroup
  \urlstyle{rm}\Url}\fi

\bibitem{RTMarler_JSArora_2004}
Marler, R.T., Arora, J.S.: Survey of multi-objective optimization methods for
  engineering.
\newblock Struct. Multidisciplinary Optim. \textbf{26}, 369--395 (2004)

\bibitem{MEhrgott_2005}
Ehrgott, M.: Multicriteria Optimization.
\newblock Springer, Heidelberg, Germany (2005)

\bibitem{ALCustodio_MEmmerich_JFAMadeira_2012}
Cust{\'o}dio, A.L., Emmerich, M., Madeira, J.F.A.: Recent developments in
  derivative-free multiobjective optimization.
\newblock Computational Technology Reviews \textbf{5}, 1--30 (2012)

\bibitem{TGKolda_RMLewis_VTorczon_2003}
Kolda, T.G., Lewis, R.M., Torczon, V.: Optimization by direct search: {N}ew
  perspectives on some classical and modern methods.
\newblock SIAM Rev. \textbf{45}, 385--482 (2003)

\bibitem{ARConn_KScheinberg_LNVicente_2009}
Conn, A.R., Scheinberg, K., Vicente, L.N.: Introduction to Derivative-free
  Optimization.
\newblock MPS-SIAM Series on Optimization. SIAM, Philadelphia, USA (2009)

\bibitem{CAudet_WHare_2017}
Audet, C., Hare, W.: Derivative-free and Blackbox Optimization.
\newblock Springer Series in Operations Research and Financial Engineering.
  Springer, Cham, Switzerland (2017)

\bibitem{ALCustodio_et_al_2011}
Cust{\'o}dio, A.L., Madeira, J.F.A., Vaz, A.I.F., Vicente, L.N.: Direct
  multisearch for multiobjective optimization.
\newblock SIAM J. Optim. \textbf{21}, 1109--1140 (2011)

\bibitem{CCartis_NIMGould_PhLToint_2012}
Cartis, C., Gould, N.I.M., \mbox{Ph}. L.~Toint: On the oracle complexity of
  first-order and derivative-free algorithms for smooth nonconvex minimization.
\newblock SIAM J. Optim. \textbf{22}, 66--86 (2012)

\bibitem{MDodangeh_LNVicente_2014}
Dodangeh, M., Vicente, L.N.: Worst case complexity of direct search under
  convexity.
\newblock Math. Program. \textbf{155}, 307--332 (2016)

\bibitem{JFliege_AIFVaz_LNVicente_2019}
Fliege, J., Vaz, A.I.F., Vicente, L.N.: Complexity of gradient descent for
  multiobjective optimization.
\newblock Optim. Methods Softw. \textbf{34}, 949--959 (2019)

\bibitem{RGarmanjani_LNVicente_2012}
Garmanjani, R., Vicente, L.N.: Smoothing and worst-case complexity for
  direct-search methods in nonsmooth optimization.
\newblock IMA J. Numer. Anal. \textbf{33}, 1008--1028 (2013)

\bibitem{GNGrapiglia_JYuan_YYuan_2015}
Grapiglia, G.N., Yuan, J., Yuan, Y.: On the convergence and worst-case
  complexity of trust-region and regularization methods for unconstrained
  optimization.
\newblock Math. Program. \textbf{152}, 491--520 (2015)

\bibitem{SGratton_et_al_2015}
Gratton, S., Royer, C.W., Vicente, L.N., Zhang, Z.: Direct search based on
  probabilistic descent.
\newblock SIAM J. Optim. \textbf{25}, 1515--1541 (2015)

\bibitem{JKonecny_PRichtarik_2014}
Kone\v{c}n{\'y}, J., Richt{\'a}rik, P.: Simple complexity analysis of
  simplified direct search.
\newblock Tech. rep., Available at https://arxiv.org/abs/1410.0390 (2014)

\bibitem{YNesterov_2004}
Nesterov, Y.: Introductory Lectures on Convex Optimization.
\newblock Applied Optimization. Kluwer Academic Publishers, Boston, USA (2004)

\bibitem{LNVicente_2013}
Vicente, L.N.: Worst case complexity of direct search.
\newblock {EURO} Journal on Computational Optimization \textbf{1}, 143--153
  (2013)

\bibitem{SGratton_ASartenaer_PhLToint_2008}
Gratton, S., Sartenaer, A., \mbox{Ph.} L.~Toint.: Recursive trust-region
  methods for multiscale nonlinear optimization.
\newblock SIAM J. Optim. \textbf{19}, 414--444 (2008)

\bibitem{CCartis_PhLSampaio_PhLToint_2015}
Cartis, C., \mbox{Ph}. R.~Sampaio, \mbox{Ph}. L.~Toint: Worst-case evaluation
  complexity of non-monotone gradient-related algorithms for unconstrained
  optimization.
\newblock Optimization \textbf{64}, 1349--1361 (2015)

\bibitem{YNesterov_BTPolyak_2006}
Nesterov, Y., Polyak, B.T.: Cubic regularization of {N}ewton method and its
  global performance.
\newblock Math. Program. \textbf{108}, 177--205 (2006)

\bibitem{CCartis_NIMGould_PhLToint_2011}
Cartis, C., Gould, N.I.M., \mbox{Ph}. L.~Toint: Adaptive cubic regularisation
  methods for unconstrained optimization. {P}art {II}: worst-case function- and
  derivative-evaluation complexity.
\newblock Math. Program. \textbf{130}, 295--319 (2011)

\bibitem{EBergou_et_al_2020}
Bergou, E.H., Gorbunov, E., Richt{\'a}rik, P.: Stochastic three points method
  for unconstrained smooth optimization.
\newblock SIAM J. Optim. \textbf{30}, 2726--2749 (2020)

\bibitem{LCalderon_et_al_2020}
Calder\'on, L., Diniz-Ehrhardt, M.A., Mart\'inez, J.M.: On high-order model
  regularization for multiobjective optimization.
\newblock Optim. Methods Softw. \textbf{online published} (2020)

\bibitem{AZilinskas_2013}
\v{Z}ilinskas, A.: On the worst-case optimal multi-objective global
  optimization.
\newblock Optim. Lett. \textbf{7}, 1921--1928 (2013)

\bibitem{JMCalvin_AZilinskas_2020}
Calvin, J.M., \v{Z}ilinskas, A.: On efficiency of a single variable
  bi-objective optimization algorithm.
\newblock Optim. Lett. \textbf{14}, 259--267 (2020)

\bibitem{KDVVillacorta_POOliveira_ASoubeyran_2014}
Villacorta, K.D.V., Oliveira, P.R., Soubeyran, A.: A trust-region method for
  unconstrained multiobjective problems with applications in satisficing
  processes.
\newblock J. Optim. Theory Appl. \textbf{160}, 865--889 (2014)

\bibitem{JFliege_BFSvaiter_2000}
Fliege, J., Svaiter, B.F.: Steepest descent methods for multicriteria
  optimization.
\newblock Math. Methods Oper. Res. \textbf{51}, 479--494 (2000)

\bibitem{EZitzler_1999}
Zitzler, E.: Evolutionary algorithms for multiobjective optimization: {M}ethods
  and applications.
\newblock Ph.D. thesis, Swiss Federal Institute of Technology Zurich,
  Switzerland (1999)

\bibitem{JThomann_GEichfelder_2019}
Thomann, J., Eichfelder, G.: A trust-region algorithm for heteregeneous
  multiobjective optimization.
\newblock SIAM J. Optim. \textbf{29}, 1017--1047 (2019)

\bibitem{JEDennis_DJWoods_1987}
\mbox{Dennis Jr.}, J.E., Woods, D.J.: Optimization on microcomputers: The
  {N}elder-{M}ead simplex algorithm.
\newblock In: A.~Wouk (ed.) New Computing Environments: Microcomputers in
  Large-Scale Computing, pp. 116--122. SIAM, Philadelphia (1987)

\bibitem{GLiuzii_et_al_2016}
Liuzzi, G., Lucidi, S., Rinaldi, F.: A derivative-free approach to constrained
  multiobjective nonsmooth optimization.
\newblock SIAM J. Optim. \textbf{26}, 2744--2774 (2016)

\end{thebibliography}
\end{document}